\documentclass[11pt,reqno]{amsart}
\usepackage{amssymb}
\usepackage{amsthm}
\usepackage{amsmath}
\usepackage{amscd}
\usepackage{verbatim,url}
\usepackage{xcolor,cancel}
\usepackage[mathscr]{eucal}
\usepackage{mathrsfs}

\usepackage[OT2,T1]{fontenc}
\DeclareSymbolFont{cyrletters}{OT2}{wncyr}{m}{n}
\DeclareMathSymbol{\Sha}{\mathalpha}{cyrletters}{"58}

\usepackage{enumitem}
\usepackage{bm}
\usepackage{cleveref}
\allowdisplaybreaks

\setlist[itemize]{leftmargin=*}
\setlist[enumerate]{leftmargin=*,label=\rm{(\arabic*)}}

\newtheorem{theorem}{Theorem}
\newtheorem{lemma}[theorem]{Lemma}

\newtheorem{proposition}[theorem]{Proposition}

\theoremstyle{remark}

\numberwithin{theorem}{section} \numberwithin{equation}{section}
\numberwithin{figure}{section}

\newcommand{\R}{\mathbb R}
\newcommand{\N}{\mathbb N}

\newcommand{\Z}{\mathbb{Z}}

\newcommand{\Pc}{\mathcal{P}}
\newcommand{\Sc}{\mathcal{S}}

\renewcommand{\a}{\alpha}
\renewcommand{\b}{\beta}

\newcommand{\e}{\varepsilon}
\newcommand{\g}{\gamma}
\renewcommand{\l}{\lambda}

\newcommand{\vf}{\varphi}

\newcommand{\s}{\sigma}
\renewcommand{\t}{\tau}
\renewcommand{\th}{\theta}
\newcommand{\vth}{\vartheta}

\newcommand{\De}{\Delta}
\newcommand{\Th}{\Theta}
\newcommand{\es}{\emptyset}
\newcommand{\sm}{\setminus}

\newcommand{\ou}{\mathrm{ou}}
\newcommand{\eu}{\mathrm{eu}}
\newcommand{\ev}{\mathrm{ev}}
\newcommand{\od}{\mathrm{od}}
\newcommand{\ed}{\mathrm{ed}}

\newcommand{\yz}{\mathrm{yz}}
\newcommand{\wx}{\mathrm{wx}}
\newcommand{\ex}{\mathrm{ex}}
\newcommand{\ez}{\mathrm{ez}}

\def\lp{\left(}
\def\rp{\right)}

\makeatletter
\@namedef{subjclassname@2020}{%
	\textup{2020} Mathematics Subject Classification}
\makeatother

\title{On the asymptotic behavior for partitions separated by parity}

\author{Kathrin Bringmann}
\author{William Craig}
\author{Caner Nazaroglu}


\begin{document}

\begin{abstract}
	The study of partitions with parts separated by parity was initiated by Andrews in connection with Ramanujan's mock theta functions, and his variations on this theme have produced generating functions with a large variety of different modular properties. In this paper, we use Ingham's Tauberian theorem to compute the asymptotic main term for each of the eight functions studied by Andrews.
\end{abstract}
\maketitle

\section{Introduction and statement of results}
A {\it partition} of a non-negative integer $n$ is a non-increasing sequence $\lambda = (\lambda_1, \lambda_2, \dots,$ $\lambda_\ell)$ of positive integers whose parts sum to $n$. Recently, Andrews \cite{A1} initiated the study of partitions with parts separated by parity. These are connected to Ramanujan's fifth order mock theta functions \cite{A2}, exhibit congruences in certain arithmetic progressions \cite{Ch}, and have recently motivated Burson and Eichhorn to introduce new generalizations of partitions called copartitions \cite{BE}, which have many interesting features. 
At its core, the rationale behind the construction of these partitions is through the 
$q$-hypergeometric series that naturally count them. 
So partitions separated by parity provide a new avenue for the exploration of combinatorial objects within the context of $q$-hypergeometric series and false, mock, or more general modular properties associated with them.

To start describing these partitions, we first introduce Andrews' notation from \cite{A1}. Each of these partition functions takes the form $p^{\mathrm{wx}}_{\mathrm{yz}}(n)$. The symbols $\mathrm{wx}, \mathrm{yz}$ are formed with first letter either $\mathrm{e}$ or $\mathrm{o}$ (which represent even or odd parts) and second character either $\mathrm{u}$ or $\mathrm{d}$ (which represent unrestricted parts or distinct parts). Parts represented by the symbol in the subscript are assumed to lie below all parts represented by the superscript. For instance, $p_{\ed}^{\ou}(n)$ counts the number of partitions with all even parts smaller than all odd parts, each even part is distinct, and odd parts are unrestricted. We consider all eight families of partitions we can build this way, namely $p_\eu^\ou(n), p_\eu^\od(n), p_\ou^\eu(n), p_\ou^\ed(n), p_\ed^\ou(n), p_\ed^\od(n), p_\od^\eu(n)$, and $p_\od^\ed(n)$. Collectively, we refer to these families as partitions with parts separated by parity.

In this paper, we consider the asymptotic properties of partitions with parts separated by parity. The study of asymptotic properties of partitions goes back to the seminal work of Hardy and Ramanujan \cite{HR}, who proved that the function $p(n)$ counting the number of partitions of $n$ satisfies
\begin{align*}
	p(n) \sim \dfrac{1}{4\sqrt{3}n} e^{\pi \sqrt{\frac{2n}{3}}}, \ \ \ \text{as } n \to \infty.
\end{align*}
Here we prove analogous asymptotics for the family of partition functions mentioned above. We use the work of Andrews \cite{A2, A1} as well as of the first author and Jennnigs-Shaffer \cite{BJ} on generating functions for these sequences.\footnote{In much of the previous literature on these topics, the generating functions for partitions with parts separated by parity (in particular $p_\ou^\eu$, $p_\ou^\ed$, $p_\od^\eu$, $p_\od^\ed$, $p_\ed^\ou$, and $p_\ed^\od$) omit certain ``trivial'' components coming from partitions into only odd parts or only even parts. In our work, these omissions are corrected.} Note that one could improve some of these partition asymptotics to include more terms or exact formulas, as the generating functions involve modular forms, mock modular forms, false modular forms, as well as false-indefinite theta functions that are hybrids of the latter two. One example in particular, namely $p_\od^\eu(n)$, is of special interest because it involves false-indefinite theta functions associated with Maass forms, which can be efficiently studied with mock Maass forms developed by Zwegers \cite{Zw}.\footnote{In upcoming work, we will study the relevant modular properties in more detail and show that a more refined asymptotic formula for $p_\od^\eu(n)$ can be obtained in this way.}

\begin{theorem} \label{T:Main Theorem}
	We have the following asymptotics as $n \to \infty$:
	\begin{align}
		p_\eu^\ou(n) &\sim \frac{e^{\pi\sqrt\frac{n}{3}}}{2\pi\sqrt{n}}, \label{T:EU-OU main term} \\
		p_\eu^\od(n) &\sim \frac{e^{\pi\sqrt\frac n3}}{4\sqrt{2} \cdot3^\frac14 n^\frac34}, \label{T:EU-OD main term} \\
		p_\od^\eu(n) &\sim \frac{e^{\pi\sqrt\frac{n}{3}}}{2\sqrt{3}n}, \label{T:OD-EU main term} \\
		p_\ed^\ou(n) &\sim \frac{e^{\pi\sqrt\frac n3}}{4\cdot3^\frac14n^\frac34}, \label{T:ED-OU main term} \\
		p_\ed^\od(n) &\sim \frac{3^\frac14\left(\sqrt2-1\right)e^{\pi\sqrt\frac n6}}{2^\frac34\pi n^\frac14}, \label{T:ED-OD main term} \\
		p_\ou^\eu(n) &\sim \frac{3^\frac14e^{\pi\sqrt\frac n3}}{2\pi n^\frac14}, \label{T:OU-EU main term} \\
		p_\ou^\ed(n) &\sim \frac{e^{\pi\sqrt\frac n3}}{4\sqrt2\sqrt n}, \label{T:OU-ED main term} \\
		p_\od^\ed(n) &\sim \frac{3^\frac14\left(\sqrt2-1\right)e^{\pi\sqrt\frac n6}}{2^\frac14\pi n^\frac14}. \label{T:OD-ED main term}
	\end{align}
\end{theorem}

The remainder of the paper is structured as follows. In Section \ref{Preliminary asymptotics}, we work out several preliminary results required for studying asymptotic properties of the generating functions we consider. In Section \ref{Combinatorial arguments}, we prove various combinatorial injections that are required in order to justify our applications of Ingham's Tauberian theorem. In Section \ref{Proofs} we prove Theorem \ref{T:Main Theorem}. Section \ref{openquestion} contains open questions.

\section*{Acknowledgements}

The first and the second author have received funding from the European Research Council (ERC) under the European Union's Horizon 2020 research and innovation programme (grant agreement No. 101001179). The first and the third author were supported by the SFB/TRR 191 “Symplectic Structure in Geometry, Algebra and Dynamics”, funded by the DFG (Projektnummer 281071066 TRR 191).

\section{Preliminaries} \label{Preliminary asymptotics}

\subsection{A Tauberian Theorem}

We compute the main terms in the relevant asymptotic expansions using the Theorem 1.1 of \cite{BJM}, which follows from Ingham's Tauberian Theorem \cite{I},
for the special case $\a=0$.
For $\Delta\geq0$ define
\begin{equation*}
	R_\Delta := \{x+iy:\, x,y\in\R,\, x>0,\, |y|\le\De x\}.
\end{equation*}	

\begin{proposition}\label{P:CorToIngham}
Let $B(q)=\sum_{n\ge0}b(n)q^n$ be a power series whose radius of convergence is at least one and assume that $b(n)$ are non-negative and weakly increasing.
Also suppose that $\l$, $\b$, $\g\in\R$ with $\g>0$ exist such that\footnote{The second condition in \eqref{E:as1} is often dropped in literature, however in its absence there are counterexamples to the conclusion of the proposition as discussed in \cite{BJM}.}
	\begin{equation}\label{E:as1}
	B\left(e^{-t}\right) \sim \l t^\b e^\frac\g t \quad\text{as } t \to 0^+,\qquad B\left(e^{-z}\right) \ll |z|^\b e^\frac{\g}{|z|} \quad\text{as } z \to 0,
	\end{equation}
	with the latter condition holding in each region $R_\Delta$
	for $\De \geq 0$. Then we have
	\[
	b(n) \sim \frac{\l\g^{\frac\b2+\frac14}}{2\sqrt\pi n^{\frac\b2+\frac34}}e^{2\sqrt{\g n}} \qquad\text{as } n \to \infty.
	\]
\end{proposition}

\subsection{A theta function}

In this subsection and throughout the rest of the paper, we define the {\it $q$-Pochhammer symbol} by
\begin{align*}
	\lp a;q \rp_n := \prod_{j=0}^{n-1} \lp 1 - a q^j \rp, \ \ \ n\in\N_0\cup\{\infty\}.
\end{align*}
The following lemma is a well-known consequence of the modularity of the \textit{Dedekind eta-function}
\begin{equation*}
	\eta\lp \tau \rp := e^{\frac{\pi i \tau}{12}} \prod_{n \geq 1} \lp 1 - e^{2\pi i n \tau} \rp.
\end{equation*}

\begin{lemma}\label{L:q-Pochhammer asymptotics}
	Let $q=e^{-z}$ and $\Delta \geq 0$. Then as $z\to0$ in $R_\Delta$, we have
	\begin{align*}
		\left( q;q \right)_\infty &\sim \sqrt{\dfrac{2\pi}{z}} e^{- \frac{\pi^2}{6z}}.
	\end{align*}
\end{lemma}

We furthermore require the Jacobi theta function
\[
	\Th(\t) := \sum_{n\in\Z} q^\frac{n^2}{2},\qquad q := e^{2\pi i\t}.
\]
We have the following asymptotic behavior as a consequence of the modularity of the Jacobi theta function.

\begin{lemma}\label{L:Jacobi asymptotics}
	Let $q=e^{-z}$ and $\Delta \geq 0$. Then as $z\to0$ in $R_\Delta$, we have
	\begin{align*}
		\Theta\left( \dfrac{iz}{2\pi} \right) \sim \sqrt{\dfrac{2\pi}{z}}.
	\end{align*}
\end{lemma}

\subsection{The Euler--Maclaurin summation formula}

In this subsection, we discuss an important framework for computing asymptotic expansions of infinite sums using the Euler--Maclaurin summation formula. A function $g$ is of {\it sufficient decay} in an unbounded domain $D \subset \mathbb{C}$ if there exists $\e>0$ such that $g(w)\ll |w|^{-1-\e}$ uniformly as $|w|\to\infty$ in $D$.
Also for $0 \leq \th < \frac{\pi}{2}$, we define
\begin{equation*}
D_\th := \{re^{i\a}:r\ge0\text{ and }|\a|\le\th\} = R_{\tan(\theta)} \cup \{0\}.
\end{equation*}
We now state the following variation of the Euler--Maclaurin summation formula from \cite{BJM,Za}. 

\begin{proposition}\label{P:EulerMaclaurin1DShifted}
Let $0\le\th<\frac\pi2$ and let $g$ be a holomorphic function in a domain containing $D_\th$ (in particular we assume that $g$ is holomorphic at the origin). Also suppose that $g$, as well as all of its derivatives, are of sufficient decay. 
Then for any $a\in\R_{\geq 0}$ and $N\in\N_0$, we have
	\[
		\sum_{m\ge0} g((m+a)z) = \frac1z\int_0^\infty g(w) dw - \sum_{n=0}^{N-1} \frac{B_{n+1}(a)g^{(n)}(0)}{(n+1)!}z^n + O \left(z^N\right),
	\]
as $z\to0$ 	uniformly in $R_{\tan(\th)}$. Here $B_n(x)$ denotes the $n$-th Bernoulli polynomial.
\end{proposition}

We also require an extension the formula above to the case of two-dimensional sums (see Proposition 5.2 of \cite{BJM}). We say that a multivariable function $f$ in two variables is of {\it sufficient decay} in $D \subset \mathbb{C}^2$ if there exist $\varepsilon_1$, $\varepsilon_2>0$ such that $f(\bm{x})\ll |w_1|^{-1-\varepsilon_1} |w_2|^{-1-\varepsilon_2}$
uniformly as $|w_1|+|w_2|\rightarrow\infty$ in $D$.

\begin{proposition}\label{P:EulerMaclaurinGeneral}
Let $0\leq \th <\frac\pi2$, let $f$ be holomorphic in a domain containing $D_{\th} \times D_\th$ (in particular $f$ is holomorphic at the origin), and suppose that $f$, along with all of its derivatives, are of sufficient decay. 
Then for any $\bm{a}\in\mathbb{R}_{\geq0}^2$ and $N\in\mathbb{N}_0$, we have
	\begin{align*}
		\sum_{\bm{m}\in\N_0^2} f((\bm{m}+\bm{a})z)
		&=
		\frac1{z^2}\int_{0}^{\infty}\!\!\int_{0}^{\infty}\! f(\bm{x}) dx_1dx_2
		-
		\frac1z\sum_{n_1=0}^{N}\frac{ B_{n_1+1}(a_1) }{(n_1+1)!} z^{n_1}
		\int_{0}^{\infty}\!  f^{(n_1,0)} (0,x_2)dx_2
		\\&\quad
		-
		\frac1z
		\sum_{n_2=0}^{N} \frac{B_{n_2+1}(a_2) }{(n_2+1)!}z^{n_2}
		\int_{0}^{\infty}\! f^{(0,n_2)}(x_1,0)dx_1
		\\&\quad
		+
		\!\sum_{n_1+n_2<N}\!\!
		\frac{B_{n_1+1}(a_1) B_{n_2+1}(a_2) f^{(n_1,n_2)}(\bm{0})}{(n_1+1)!(n_2+1)!}z^{n_1+n_2}
		+
		O\left(z^N\right)
	\end{align*}
	as $z \to 0$ uniformly in $R_{\tan(\th)}$.
\end{proposition}

\subsection{A mock theta function of Ramanujan}
Next we require the asymptotic behavior of Ramanujan's third order mock theta function (see \cite[p. 64]{W1})
\begin{align}\nonumber
	f(q) &:= \sum_{n\ge0} \frac{q^{n^2}}{(-q;q)_n^2} = \frac{2}{(q;q)_\infty}\sum_{n\in\Z} \frac{(-1)^nq^\frac{n(3n+1)}{2}}{1+q^n}\\
	\label{f-definition}
	&\hspace{.1cm}= 2 - \frac{2}{(q;q)_\infty}\sum_{n\in\Z} \frac{(-1)^nq^\frac{3n(n+1)}{2}}{1+q^n}.
\end{align}
We have the following asymptotic behavior.

\begin{lemma} \label{L:f lemma}
Let $q=e^{-z}$ and $\Delta \geq 0$. Then as $z\to0$ in $R_\Delta$, we have
	\begin{align*}
		f(-q) \sim - \sqrt{\dfrac{\pi}{z}} \, e^{\frac{\pi^2}{24z}}.
	\end{align*}
\end{lemma}

\begin{proof}
We have the following identity due to Watson \cite[p. 63]{W1}:
	\[
		f(-q) = 2\phi(q) - \frac{\vth_4(0;-q)}{(q;-q)_\infty},
	\]
	with the mock theta function
	\begin{align*}
		\phi(q) := \sum_{n \geq 0} \dfrac{q^{n^2}}{\left( -q^2; q^2 \right)_n}
	\end{align*}
	and the theta function
	\[
		\vth_4(w;q) := \sum_{n\in\Z} (-1)^ne^{2\pi inw}q^{n^2}.
	\]
	Now, $\vth_4(0;-q)=\sum_{n\in\Z}q^{n^2}=\Th(2\t)$ with $q=e^{2\pi i\t}$, and so using \Cref{L:Jacobi asymptotics} we obtain\footnote{
Here and in the rest of the paper we state asymptotic results as $z\to0$ in $R_\Delta$ if not explicitly stated otherwise.
	}
	\begin{align*}
		\vth_4(0; -q) \sim \sqrt{\dfrac{\pi}{z}}.
	\end{align*}
	We rewrite
	\[
		(q;-q)_\infty = \left(-q^2;q^2\right)_\infty\left(q;q^2\right)_\infty = \frac{(q;q)_\infty\left(q^4;q^4\right)_\infty}{\left(q^2;q^2\right)_\infty^2}.
	\]
	Therefore using \Cref{L:q-Pochhammer asymptotics} we have
	\begin{align} \label{E:(q;-q)}
		\left( q; -q \right)_\infty \sim e^{- \frac{\pi^2}{24z}}.
	\end{align}
	Combining the asymptotics above, we obtain
	\begin{align*}
		- \dfrac{\vartheta_4\left( 0;-q \right)}{\left( q; -q \right)_\infty} \sim - \sqrt{\dfrac{\pi}{z}} e^{\frac{\pi^2}{24z}}.
	\end{align*}
	
	We next show that the contribution of $\phi(q)$ is polynomially smaller.
	Using \cite[equations (1.9), (1.10), (1.7), (1.8)]{Mo}, we write
	\begin{equation*}
		\phi(q) = \Phi_{0}\left(q^{2}\right) + q\, \Phi_{1}\!\left(q^{2}\right),
	\end{equation*}
	where
	\begin{align*}
		\Phi_{0}(q) &:=\frac{1}{\left(q^{2};q^{2}\right)_{\infty}} \sum_{\substack{n\geq0\\-n\leq j\leq n}} (-1)^{j}q^{4n^{2}+n-j^{2}}\left(1-q^{6n+3}\right),\\
		\Phi_{1}(q) &:=\frac{1}{\left(q^{2};q^{2}\right)_{\infty}} \sum_{\substack{n\geq0\\-n\leq j\leq n}} (-1)^{j}q^{4n^{2}+3n-j^{2}}\left(1-q^{2n+1}\right).
	\end{align*}
	By \Cref{L:q-Pochhammer asymptotics}, we have
	\begin{equation*}
		\frac{1}{\left(q^{4};q^{4}\right)_{\infty}} \sim \sqrt{\frac{2z}{\pi}}e^{\frac{\pi^{2}}{24z}}.
	\end{equation*}
	We now consider $\Phi_j^*(q):=(q^2;q^2)_\infty\Phi_j(q)$ ($j\in\{0,1\}$). We have
	\begin{equation*}
		\Phi_0^*(q) = 2\sum_{0\leq j\leq n} (-1)^j q^{4n^{2}+n-j^{2}}\left(1-q^{6n+3}\right) - \sum_{n\geq0}q^{4n^{2}+n}\left(1-q^{6n+3}\right).
	\end{equation*}
	The second summand is $O(\frac{1}{\sqrt{z}})$ by \Cref{P:EulerMaclaurin1DShifted}. We write the first summand as 
	\begin{equation*}
		2q^{-\frac{1}{16}} \sum_{\delta\in\{0,1\}} (-1)^\delta \sum_{n,j\geq0}\left(f\left(\left( n+\frac{1}{8},j+\frac{\delta}{2}\right)w\right)-f\left(\left(n+\frac{7}{8},j+\frac{\delta}{2}\right)w\right)\right),
	\end{equation*}
	where $f(\bm x) := e^{-4x_1^2 -16x_1x_2-12x_2^2}$ and $w = \sqrt{z}$.
	Using \Cref{P:EulerMaclaurinGeneral} for $w$ in $R_{\tan (\th)}$ with $0 \leq \th <\frac{\pi}{4}$ gives that this is $O(\frac1{\sqrt z})$ in all regions of the form $R_\Delta$ for $\Delta \geq 0$. The contribution of $\Phi_1^*$ is treated in the same way.
\end{proof}
\subsection{The function $\sigma$}
Finally we consider Ramanujan's $\sigma$-function \cite{ADH}
\begin{equation*}
	\sigma(q) := \sum_{n \geq 0} \dfrac{q^{\frac{n(n+1)}{2}}}{\lp -q;q \rp_n}.
\end{equation*}
We have the following asymptotic behavior.
\begin{lemma}\label{L:S}
Let $q=e^{-z}$ and $\Delta \geq 0$. Then as $z\to0$ in $R_\Delta$, we have
	\begin{equation*}
		\sigma\left(-e^{-z}\right) \sim -2.
	\end{equation*}
\end{lemma}
\begin{proof}
	By Theorem 1 of \cite{ADH}, we have
	\begin{align*}
	\s(q) = \sum_{\substack{n\geq0\\-n\leq j\leq n}}	\left(-1\right)^{n+j}\left(1-q^{2n+1}\right)q^{\frac{n\left(3n+1\right)}{2}-j^2}.
	\end{align*}
	We write this as
	\begin{equation}\label{E:RewriteS}
	\s(q) = 2 \sum_{n\ge j\ge0} (-1)^{n+j}\left(1-q^{2n+1}\right)q^{\frac{n(3n+1)}{2}-j^2} - \sum_{n \geq 0} (-1)^n\left(1-q^{2n+1}\right)q^\frac{n(3n+1)}{2}.
	\end{equation}
	We may then rewrite
	\begin{align*}
		\sigma(-q) &= 2q^{-\frac{1}{24}} \sum_{0\leq a,b\leq 3}(-1)^{\frac{a(a+1)}{2}+\frac{b(b+1)}{2}+ab}\\
		&\quad\times\sum_{\bm{n}\in\N_{0}^{2}} \left(f\left(\left(n_{1}+\frac{a}{4}+\frac{1}{24},n_{2}+\frac{b}{4}\right)w\right)+f\left(\left(n_1+\frac{a}{4}+\frac{5}{24},n_{2}+\frac{b}{4}\right)w\right)\right)\\
		&\quad-q^{-\frac{1}{24}} \sum_{0\leq a\leq 3} (-1)^{\frac{a(a+1)}{2}} 
		\sum_{n\geq0}\left(g\left(\left(n+\frac{a}{4}+\frac{1}{24}\right)w\right)+ g\left(\left(n+\frac{a}{4}+\frac{5}{24}\right)w\right)\right),
	\end{align*}
	where $w := \sqrt{z}$ and
	\begin{equation*}
		f(\bm{x}):=e^{-24x_{1}^{2}-8x_{2}^{2}-48x_{1}x_{2}}, \qquad g(x):=e^{-24x^{2}}.
	\end{equation*}
	Using \Cref{P:EulerMaclaurinGeneral}, the claim now follows by a direct calculation.
\end{proof}
\section{Combinatorial arguments} \label{Combinatorial arguments}

In this section we show that $p_\yz^\wx(n)$ is weakly increasing for six of the eight cases. In the other two cases (namely $p_\eu^\od(n)$ and $p_\ed^\od(n)$) we prove that $p_\yz^\wx(2n)$ and $p_\yz^\wx(2n+1)$ are weakly increasing. Throughout this section, we use the notation $\Pc_\yz^\wx(n)$ for the set of all partitions of $n$ counted by $p_\yz^\wx(n)$. We begin with the following lemma, which applies to all eight cases. In practice, we only apply this lemma to $p_\eu^\od(n)$ and $p_\ed^\od(n)$.

\begin{lemma} \label{Easy Injection}
	We have $p_\yz^\wx(n) \leq p_\yz^\wx(n+2)$ for $n\in\N_0$.
\end{lemma}

\begin{proof}
	Consider the map $\vf:\Pc_\yz^\wx(n) \to \Pc_\yz^\wx(n+2)$ defined by $\vf(\es)=(2)$ if $n=0$ and 
	\begin{equation*}
	 (\l_1,\l_2,\dots,\l_\ell) \mapsto (\l_1+2,\l_2,\dots,\l_\ell)
	\end{equation*}
	otherwise.
	It is clear that $\vf(\l)$ does not violate any distinctness conditions. Since $\l_1,\l_1+2$ have the same parity, the parities of $\vf(\l)$ remain separated and therefore $\vf$ is well-defined. It is clear that $\vf$ is invertible on its image, so $\vf$ is injective and thus $p_\yz^\wx(n)\le p_\yz^\wx(n+2)$.
\end{proof}

A short computation shows that $p_\eu^\od(n)$ and $p_\ed^\od(n)$ are not weakly increasing sequences; we have $p_\eu^\od(4) = 3 > 2 = p_\eu^\od(5)$ and $p_\ed^\od(6) = 3 > 2 = p_\ed^\od(7)$. Therefore in these two cases, \Cref{Easy Injection} is the strongest information we can obtain. We show in the following lemmas that the other six cases give weakly increasing sequences.

\begin{lemma}\label{Injection}
	We have that $p_\ou^\eu(n)$, $p_\ou^\ed(n)$, $p_\ed^\ou(n)$, $p_\eu^\ou(n)$, $p_\od^\eu(n)$, and $p_\od^\ed(n)$ are weakly increasing.
\end{lemma}

\begin{proof}
	This proof splits into three cases; we construct three injective maps, each of which establishes two of the six cases we consider.
	
	We first consider the cases $p_\ou^\eu(n)$ and $p_\ou^\ed(n)$. For $\mathrm x\in\{\mathrm d,\mathrm u\}$, define the map 
	\[
		\vf_1:\Pc_\ou^\ex(n) \to \Pc_\ou^\ex(n+1),\qquad \l \mapsto \l\cup(1).
	\]
	Since odd parts are unrestricted and 1 is less than all positive even numbers, we have $\vf_1(\l)\in\Pc_\ou^\ex(n+1)$. It is also clear that $\vf_1$ is invertible on its image. Therefore, $\vf_1$ is injective and $p_\ou^\ex(n)$ is weakly increasing for $\mathrm x\in\{\mathrm d,\mathrm u\}$.
	
	We now consider the cases  $p_\ed^\ou(n)$ and $p_\eu^\ou(n)$. For $\mathrm z\in\{\mathrm d,\mathrm u\}$, we construct injections $\vf_2:\Pc_\ez^\ou(n)\to\Pc_\ez^\ou(n+1)$. For a partition $\l\in\Pc_\ez^\ou$, we write the partition with slightly adjusted notation as $\l=(\l_1,\dots,\l_{m-1},\l_m,\dots,\l_\ell)$, and we now choose $m=m(\l)$ such that $\l_m$ is the largest even part (in the case of repetitions, take $m$ so that $\l_{m-1}$ is odd), if indeed such an $m$ exists (i.e., if $\l$ has any even parts). We define $\vf_2$ by
	\[
	\vf_2(\l) :=
	\begin{cases}
	\l\sm(\l_m)\cup(\l_m+1) & \text{if }\l\text{ has an even part},\\
	\l\cup(1) & \text{otherwise}.
	\end{cases}
	\]
	We first show that $\vf_2$ is well-defined. It is clear that $|\vf_2(\l)|=n+1$ for $\l\in\Pc_\ez^\ou(n)$. If $\l$ has no even parts, then $\l$ and $\vf_2(\l)$ are both partitions into odd parts, and thus in this case $\vf_2(\l)\in\Pc_\ez^\ou(n+1)$. If on the other hand $\l$ has an even part, then the action of $\vf_2$ on $\l$ changes the largest even part $\l_m$ into the odd part $\l_m+1$. It is clear that all even parts (if any) of $\vf_2(\l)$ are smaller than all odd parts of $\vf_2(\l)$, and if $\l$ has no repeated even parts, this also clearly holds for $\vf_2(\l)$ since no new even parts are added. Thus in this case $\vf_2(\l)\in\Pc_\ez^\ou(n+1)$, and therefore $\vf_2$ is well-defined.
	
	We next show that $\vf_2$ is injective. It is clear by inspection that each subcase of $\vf_2$ is a invertible process, and therefore it would suffice to show that if $\l,\mu\in\Pc_\ez^\ou(n)$ are such that $\l$ has an even part and $\mu$ does not have an even part, $\vf_2(\l)\ne\vf_2(\mu)$. Note that since $\l$ has an even part and since all odd parts are required to be larger than even parts in $\lambda$, we must have $1\notin\l$, and therefore $1\notin\vf_2(\l)$. Since clearly $1\in\vf_2(\mu)$, we have $\vf_2(\l)\ne\vf_2(\mu)$ as claimed. Thus $\vf_2$ is injective and $p_\ez^\ou(n)$ is weakly increasing for $\mathrm z\in\{\mathrm d,\mathrm u\}$.

We finally consider the cases $p_\od^\eu(n)$ and $p_\od^\ed(n)$. It is not hard to see that $p_\od^\ex(0) = p_\od^\ex(1) = p_\od^\ex(2) = 1$, so we need only consider the cases $n \geq 2$.
Let $\l_1$ denote the largest part of $\l$ (note that for $n\ge2$, necessarily $\l_1 > 1$) and let $\mathrm x \in \{ \mathrm d, \mathrm u \}$.
Consider the map $\vf_3:\Pc_\od^\ex(n)\to\Pc_\od^\ex(n+1)$ given by
	\[
	\vf_3(\l) :=
	\begin{cases}
	\l\cup(1) & \text{if }1\notin\l,\\
	\l\sm(1,\l_1)\cup(\l_1+2) & \text{if }1\in\l.
	\end{cases}
	\]	
Note that $\vf_3(\l)$ is a partition of $n+1$ and all even parts remain greater than all odd parts. Moreover, all odd parts of $\vf_3(\l)$ remain distinct, both in the first case and also in the second case if $\l_1$ is odd (i.e.~if there are no even parts) since $\l_1+2$ is larger than all the numbers in $\l$. Similarly, if $\l_1$ is even, a possible distinctness condition is preserved by the addition of $\l_1 + 2$. Therefore, $\vf_3$ is well-defined. 
Next we note that $\vf_3$ is an invertible process in each subcase. Furthermore, noting that $1\in\vf_3(\l)$ if and only if $1\notin\l$, the two subcases have disjoint images and so $\vf_3$ is injective. Therefore, $p_\od^\ex(n)$ is weakly increasing for $\mathrm x \in \{ \mathrm d, \mathrm u \}$.
\end{proof}

\section{Proof of Theorem \ref{T:EU-OU main term}} \label{Proofs}

We determine the asymptotic behavior for each $p_\yz^\wx(n)$ using \Cref{P:CorToIngham}. Throughout this section, we use the notation $F_\yz^\wx (q)$ for the generating function
\begin{equation*}
F_\yz^\wx (q) := \sum_{n\ge0} p_\yz^\wx(n) q^n.
\end{equation*}

\subsection{Proof of equation (\ref{T:EU-OU main term})}

By \Cref{Injection}, $p_\eu^\ou(n)$ is weakly increasing, so we may apply \Cref{P:CorToIngham}. Now, recall that the generating function
\begin{multline*}
F_\eu^\ou(q) = 1 + q + 2 q^2 + 2 q^3 + 4 q^4 + 4 q^5 + 7 q^6 + 7 q^7 + 12 q^8 + 
 12 q^9 + 19 q^{10} \\ + 19 q^{11} + 30 q^{12} + 30 q^{13} + 45 q^{14} + 45 q^{15} +
 67 q^{16} + 67 q^{17} + 97 q^{18} +  \ldots
\end{multline*}
is given by \cite[equation (2.1)]{A2}
\begin{align*}
	F_\eu^\ou(q) = \dfrac{1}{(1-q)\left( q^2; q^2 \right)_\infty}.
\end{align*}
For $q=e^{-z}$, we have that, as $z\to0$,
\[
	\frac{1}{1-q} \sim \frac1z.
\]
Utilizing \Cref{L:q-Pochhammer asymptotics}, we thus have, as $z\to0$ in regions $R_\Delta$,
\[
	F_\eu^\ou(q) \sim \frac{e^\frac{\pi^2}{12z}}{\sqrt{\pi z}}.
\]
Applying \Cref{P:CorToIngham} with $\l=\frac{1}{\sqrt\pi}$, $\b=-\frac12$, and $\g=\frac{\pi^2}{12}$ gives equation (\ref{T:EU-OU main term}).

\subsection{Proof of equation (\ref{T:EU-OD main term})}

By \Cref{Easy Injection}, $p_\eu^\od(2n)$ and $p_\eu^\od(2n+1)$ are both weakly increasing. In order to apply Proposition \ref{P:CorToIngham}, we need the generating functions for $p_\eu^\od(2n)$ and $p_\eu^\od(2n+1)$. First, recall the generating function 
\begin{multline*}
F_\eu^\od(q) = 
1 + q + q^2 + q^3 + 3 q^4 + 2 q^5 + 4 q^6 + 3 q^7 + 7 q^8 + 6 q^9 + 
 10 q^{10} \\ + 8 q^{11} + 16 q^{12} + 13 q^{13} + 22 q^{14} + 18 q^{15} + 34 q^{16} + 
 27 q^{17} + 46 q^{18} + \ldots 
\end{multline*}
is equal to (see \cite[equation (3.1)]{A2})
\begin{align*}
	F_\eu^\od(q) = \dfrac{1}{\left( q^2; q^2 \right)_\infty} \sum_{n \geq 0} q^{n^2}.
\end{align*}
We have the generating functions
\begin{align*}
	G_\ev(q) &:= \frac{F_\eu^\od(q)+F_\eu^\od(-q)}{2} = \sum_{n\ge0} p_\eu^\od(2n)q^{2n} = \frac{1}{\left(q^2;q^2\right)_\infty}\sum_{n\ge0} q^{4n^2},\\
	G_\od(q) &:= \frac{F_\eu^\od(q)-F_\eu^\od(-q)}{2} = \sum_{n\ge0} p_\eu^\od(2n+1)q^{2n+1} = \frac{1}{\left(q^2;q^2\right)_\infty}\sum_{n\ge0} \left(q^{n^2}-q^{4n^2}\right).
\end{align*}
We apply \Cref{P:CorToIngham} to $G_\ev(q^\frac12)$ and $q^{-\frac12}G_\od(q^\frac12)$. Short calculations reveal that
\[
	G_\ev\left(q^\frac12\right) = \frac{\Th(4\t)+1}{2(q;q)_\infty},\qquad q^{-\frac12}G_\od\left(q^\frac12\right) = \frac{q^{-\frac12}(\Th(\t)-\Th(4\t))}{2(q;q)_\infty}.
\]
Using Lemmas \ref{L:q-Pochhammer asymptotics} and \ref{L:Jacobi asymptotics}, we obtain for $q = e^{-z}$ that as $z \to 0$ in regions $R_\Delta$,
\[
	G_\ev\left(q^\frac12\right) \sim \frac{e^\frac{\pi^2}{6z}}{4},\qquad q^{-\frac12}G_\od\left(q^\frac12\right) \sim \frac{e^\frac{\pi^2}{6z}}{4}.
\]
Applying \Cref{P:CorToIngham} with $\b=0,\g=\frac{\pi^2}{6}$, and $\l=\frac14$ gives 
\begin{equation*}
	p_\eu^\od(2n), \  p_\eu^\od(2n+1) \sim \frac{e^{\pi \sqrt{\frac{2n}{3}}}}{8 \cdot 6^{\frac{1}{4}} n^{\frac{3}{4}}}.
\end{equation*}
This implies equation (\ref{T:EU-OD main term}).

\subsection{Proof of equation (\ref{T:OD-EU main term})}
First, we note that by \Cref{Injection}, $p_\od^\eu(n)$ is weakly increasing, and therefore we can apply \Cref{P:CorToIngham} to obtain asymptotics for $p_\od^\eu(n)$, so we must compute asymptotics for $F_\od^\eu(q)$ as $q\to1$. 
For the generating function
\begin{multline*}
F_\od^\eu(q) = 
1 + q + q^2 + 2 q^3 + 3 q^4 + 3 q^5 + 4 q^6 + 5 q^7 + 8 q^8 + 8 q^9 + 
 10 q^{10} \\ + 12 q^{11} + 17 q^{12} + 17 q^{13} + 22 q^{14} + 26 q^{15} + 
 34 q^{16} + 35 q^{17} + 44 q^{18} + \ldots,
\end{multline*}
we have by \cite[p. 2]{BJ} that\footnote{The corresponding formula in \cite{A1} has a missing sign in addition to the missing ``trivial'' term.} 
\[
	F_\od^\eu(-q) = - \frac{1}{\left(q^2;q^2\right)_\infty}\sum_{n \ge j\ge1} (-1)^{n+j}\left(1-q^{2n+1}\right)q^{\frac{n(3n+1)}{2}-j^2} + \frac{1}{\left(q^2;q^2\right)_\infty}.
\]
In the second term we change $n\mapsto-n-1$ for the second summand involving the term $q^{2n+1}$ and note that we obtain the first summand. Thus by Euler's pentagonal identity the second term is (see for example \cite[Theorem 1.60]{O})
\begin{align*}
	\sum_{n\in\Z} (-1)^nq^\frac{n(3n+1)}{2} = (q;q)_\infty.
\end{align*}
Using \eqref{E:RewriteS}, we obtain
\begin{align*}
	F_\od^\eu(-q) = - \frac1{2\left(q^2;q^2\right)_\infty}(\s(q)-(q;q)_\infty) + \frac{1}{\left(q^2;q^2\right)_\infty}.
\end{align*}
Thus
\begin{align*}
	F_\od^\eu(q) = \frac1{\left(q^2;q^2\right)_\infty} \left(1 - \frac{\s(-q)}2 + \frac{(-q;-q)_\infty}{2}\right).
\end{align*}
Now \Cref{L:q-Pochhammer asymptotics} gives that, as $z \to 0$ in regions $R_\Delta$, we have
\begin{equation} \label{E:(q^2;q^2)_(-1)}
	\frac1{\left(q^2;q^2\right)_\infty} \sim \sqrt{\frac{z}{\pi}} e^{\frac{\pi^2}{12z}}.
\end{equation}
Using elementary identities with $q$-Pochhammer symbols and Lemma \ref{L:q-Pochhammer asymptotics}, we obtain
\[
	(-q;-q)_\infty = \frac{\left(q^2;q^2\right)_\infty^3}{(q;q)_\infty\left(q^4;q^4\right)_\infty} \sim \sqrt{\dfrac{\pi}{z}}  e^{-\frac{\pi^2}{24z}}
\]
as $z \to 0$ in $R_\Delta$. Thus, by \Cref{L:S},
\begin{align*}
	F_\od^\eu(q) \sim 2\sqrt{\frac{z}{\pi}} e^{\frac{\pi^2}{12z}}.
\end{align*}
Now we may use \Cref{P:CorToIngham} with $\l=\frac2{\sqrt\pi}$, $\b=\frac12$, and $\g=\frac{\pi^2}{12}$, to obtain equation (\ref{T:OD-EU main term}).

\subsection{Proof of equation (\ref{T:ED-OU main term})}
For the generating function
\begin{multline*}
F_\ed^\ou(q) =
1 + q + 2 q^2 + 2 q^3 + 3 q^4 + 4 q^5 + 6 q^6 + 6 q^7 + 9 q^8 + 
 10 q^9 + 14 q^{10} \\ + 16 q^{11} + 21 q^{12} + 23 q^{13} + 31 q^{14} + 34 q^{15} + 
 44 q^{16} + 49 q^{17} + 62 q^{18} + \ldots,
\end{multline*}
recall the identity \cite[p. 2]{BJ}
\[
	F_\ed^\ou(-q) = \frac{1}{2\left(-q;q^2\right)_\infty}\left((-q;q)_\infty+1-\sum_{n\ge0} \left(1-q^n\right)q^\frac{n(3n-1)}{2}\right).
\]
Thus,
\[
	F_\ed^\ou(q) = \frac{1}{2\left(q;q^2\right)_\infty}\left((q;-q)_\infty+1-\sum_{n\ge0} \left(1-(-1)^nq^n\right)(-1)^\frac{n(3n-1)}{2}q^\frac{n(3n-1)}{2}\right).
\]
Using Lemma \ref{L:q-Pochhammer asymptotics}, we have as $z \to 0$ in $R_\Delta$, the asymptotic
\begin{equation}\label{E:(q;q^2)}
	\frac{1}{\left(q;q^2\right)_\infty} = \frac{\left(q^2;q^2\right)_\infty}{(q;q)_\infty} \sim \frac{e^\frac{\pi^2}{12z}}{\sqrt2}.
\end{equation}
We also recall \eqref{E:(q;-q)}. Finally, we apply Euler--Maclaurin summation to the remaining component of $F_\ed^\ou(q)$. For this we write
\begin{multline*}
	-\sum_{n\ge0} \left(1-(-1)^nq^n\right)(-1)^\frac{n(3n-1)}{2}q^\frac{n(3n-1)}{2}\\
	= q^{-\frac{1}{24}}\sum_{n\ge1} \left((-1)^\frac{n(3n+1)}{2}q^{\frac32\left(n+\frac16\right)^2} + (-1)^{\frac{n(3n-1)}{2}+1}q^{\frac32\left(n-\frac16\right)^2}\right).
\end{multline*}
Decomposing $n = 4m + \delta$ with $m\in\N$ and $1 \leq \delta \leq 4$, the above becomes
\begin{equation*}
	q^{-\frac{1}{24}} \sum_{\alpha \in \mathcal S} \varepsilon(\alpha) \sum_{m \geq 0} f\left(  \left( m + \alpha \right) \sqrt{z}\right),
\end{equation*}
where we define
\begin{align*}
	f(x) := e^{-24x^2}, \ \ \ \mathcal S := \left\{ \frac{5}{24}, \frac{7}{24}, \frac{11}{24}, \frac{13}{24}, \frac{17}{24}, \frac{19}{23}, \frac{23}{24}, \frac{25}{24} \right\},
\end{align*}
and
\begin{align*}
	\varepsilon(\alpha) := \begin{cases} 1 & \text{ if } \alpha \in \left\{ \frac{5}{24}, \frac{7}{24}, \frac{11}{24}, \frac{25}{24} \right\}, \\ -1 & \text{ if } \alpha \in \left\{ \frac{13}{24}, \frac{17}{24}, \frac{19}{24}, \frac{23}{24} \right\}. \end{cases}
\end{align*}
Using \Cref{P:EulerMaclaurin1DShifted} with $N=1$, we obtain by a straightforward calculation that
\[
	 q^{-\frac{1}{24}}\sum_{\a\in\Sc} \e(\a)\sum_{m\ge0} f\left((m+\a)\sqrt z\right) \sim \frac{1}{\sqrt z}\int_0^\infty f(w) dw \sum_{\a\in\Sc} \e(\a) - \sum_{\a\in\Sc} \e(\a)B_1(\a) = 1,
\]
where we use that $\sum_{\a\in\Sc}\e(\a)=0$.

Combining the calculations above, we obtain
\[
	F_\ed^\ou(q) \sim \frac{e^\frac{\pi^2}{12z}}{\sqrt2}
\]
as $z \to 0$ in $R_\Delta$. Therefore, by applying \Cref{P:CorToIngham} with $\l=\frac{1}{\sqrt2}$, $\b=0$, and $\g=\frac{\pi^2}{12}$, we obtain equation (\ref{T:ED-OU main term}).

\subsection{Proof of equation (\ref{T:ED-OD main term})}
For the generating function
\begin{multline*}
F_\ed^\od(q) = 
1 + q + q^2 + q^3 + 2 q^4 + 2 q^5 + 3 q^6 + 2 q^7 + 4 q^8 + 4 q^9 + 
 6 q^{10} \\ + 5 q^{11} + 8 q^{12} + 7 q^{13} + 10 q^{14} + 9 q^{15} + 14 q^{16} + 
 13 q^{17} + 18 q^{18} + \ldots,
\end{multline*}
recall the identity \cite[Theorem 1.1]{BJ}
\begin{align*}
	F_\ed^\od(q) = \dfrac{\left( -q; q^2 \right)_\infty}{1-q} - \dfrac{q \left( -q^2; q^2 \right)_\infty}{1-q}.
\end{align*}
Although $p_\ed^\od(n)$ is not weakly increasing, we have from \Cref{Easy Injection} that both $p_\ed^\od(2n)$ and $p_\ed^\od(2n+1)$ are weakly increasing. We therefore apply \Cref{P:CorToIngham} to the generating functions of these sequences. For this, we define
\begin{align*}
	G_\ev(q) &:= \frac{F_\ed^\od(q)+F_\ed^\od(-q)}{2}\\
	&\hspace{.1cm}= \frac12\left(\frac{\left(-q;q^2\right)_\infty}{1-q} - \frac{q\left(-q^2;q^2\right)_\infty}{1-q} + \frac{\left(q;q^2\right)_\infty}{1+q} + \frac{q\left(-q^2;q^2\right)_\infty}{1+q}\right);\\
	G_\od(q) &:= \frac{F_\ed^\od(q)-F_\ed^\od(-q)}{2}\\
	&\hspace{.1cm}= \frac12\left(\frac{\left(-q;q^2\right)_\infty}{1-q} - \frac{q\left(-q^2;q^2\right)_\infty}{1-q} - \frac{\left(q;q^2\right)_\infty}{1+q} - \frac{q\left(-q^2;q^2\right)_\infty}{1+q}\right).
\end{align*}
It is not hard to see that
\begin{align*}
	G_\ev\left( q^{\frac 12} \right) = \sum_{n \geq 0} p_\ed^\od(2n) q^n, \ \ \ q^{-\frac 12} G_\od\left( q^{\frac 12} \right) = \sum_{n \geq 0} p_\ed^\od(2n+1) q^n.
\end{align*}
We may now apply Proposition \ref{P:CorToIngham} to obtain asymptotics for $p_\ed^\od(2n)$ and $p_\ed^\od(2n+1)$. First, we use Lemma \ref{L:q-Pochhammer asymptotics} to obtain as $z \to 0$ in $R_\Delta$
\begin{align} \label{E:(-q;q^2)}
	\left( -q; q^2 \right)_\infty = \dfrac{\left( q^2; q^2 \right)_\infty^2}{\left( q;q \right)_\infty \left( q^4; q^4 \right)_\infty} \sim e^{\frac{\pi^2}{24z}}
\end{align}
and note $\frac{1}{1-q} \sim \frac{1}{z}$ to get
\[
	\frac12\frac{\left(-q;q^2\right)_\infty}{1-q} \sim \frac{e^\frac{\pi^2}{24z}}{2z}
\]
as $z \to 0$ in $R_\Delta$. 
Likewise using \Cref{L:q-Pochhammer asymptotics}, we note that as $z\to0$ in $R_\Delta$ we have
\begin{equation}\label{E:(-q^2;q^2)}
	\left(-q^2;q^2\right)_\infty = \frac{\left(q^4;q^4\right)_\infty}{\left(q^2;q^2\right)_\infty} \sim \frac{e^\frac{\pi^2}{24z}}{\sqrt2}.
\end{equation}
Using this result and \eqref{E:(q;q^2)}, respectively, we obtain as $z \to 0$ in $R_\Delta$ that
\[
	\frac12\left(-\frac{q\left(-q^2;q^2\right)_\infty}{1-q} \pm \frac{q\left(-q^2;q^2\right)_\infty}{1+q}\right) \sim -\frac{e^\frac{\pi^2}{24z}}{2\sqrt2z},\qquad \pm \frac12\frac{\left(q;q^2\right)_\infty}{1+q} \sim \pm \frac{1}{2\sqrt{2}} e^{-\frac{\pi^2}{12z}}.
\]
Combining these asymptotic formulas, we have as $z \to 0$ in $R_\Delta$ that
\begin{align*}
	G_\ev\left( q^{\frac 12} \right) \sim q^{-\frac 12} G_\od\left( q^{\frac 12} \right) \sim \left( 1 - \dfrac{1}{\sqrt{2}} \right) \dfrac{e^{\frac{\pi^2}{12z}}}{z}.
\end{align*}
By applying Proposition \ref{P:CorToIngham} with $\lambda = 1 - \frac{1}{\sqrt{2}}$, $\beta = -1$, and $\gamma = \frac{\pi^2}{12}$, we obtain
\[
	p_\ed^\od(2n) \sim p_\ed^\od(2n+1) \sim \frac{3^\frac14\left(\sqrt2-1\right)e^{\pi\sqrt\frac n3}}{2\pi n^\frac14},
\]
which completes the proof of equation (\ref{T:ED-OD main term}) if we replace $2n,2n+1$ by $n$.

\subsection{Proof of equation (\ref{T:OU-EU main term})}
For the generating function
\begin{multline*}
F_\ou^\eu(q) =
1 + q + 2 q^2 + 3 q^3 + 5 q^4 + 6 q^5 + 10 q^6 + 12 q^7 + 18 q^8 + 
 21 q^9 + 31 q^{10} \\ + 36 q^{11} + 51 q^{12} + 58 q^{13} + 80 q^{14} + 92 q^{15} + 
 124 q^{16} + 140 q^{17} + 186 q^{18} + \ldots,
\end{multline*}
from \cite[equation (2.2)]{A1}, we have
\[
	F_\ou^\eu(q) = \frac{1}{1-q}\left(\frac{1}{\left(q;q^2\right)_\infty} - \frac{q}{\left(q^2;q^2\right)_\infty}\right).
\]
Using Lemma \ref{L:q-Pochhammer asymptotics} we have the asymptotics
\[
	\frac{1}{\left(q;q^2\right)_\infty} = \frac{\left(q^2;q^2\right)_\infty}{(q;q)_\infty} \sim \frac{e^\frac{\pi^2}{12z}}{\sqrt2},
\]
as $z \to 0$ in any region $R_\Delta$, and using this and \eqref{E:(q^2;q^2)_(-1)} we obtain for $z \to 0$ for $R_\Delta$ that
\[
	F_\ou^\eu(q) \sim \frac{e^\frac{\pi^2}{12z}}{\sqrt2z}.
\]
Using \Cref{P:CorToIngham} with $\l=\frac{1}{\sqrt2}$, $\b=-1$, and $\g=\frac{\pi^2}{12}$, we obtain equation (\ref{T:OU-EU main term}).

\subsection{Proof of equation (\ref{T:OU-ED main term})}

By \Cref{Injection}, $p_\ou^\ed(n)$ is weakly increasing. Therefore, we may compute its asymptotic behavior using \Cref{P:CorToIngham} once we determine suitable asymptotics for $F_\ou^\ed(q)$. To accomplish this, we recall that for the generating function 
\begin{multline*}
F_\ou^\ed(q) = 
1 + q + 2 q^2 + 3 q^3 + 4 q^4 + 5 q^5 + 8 q^6 + 10 q^7 + 13 q^8 + 
 16 q^9 + 22 q^{10} \\ + 26 q^{11} + 34 q^{12} + 41 q^{13} + 52 q^{14} + 62 q^{15} + 
 78 q^{16} + 91 q^{17} + 113 q^{18} + \ldots,
\end{multline*}
we have the following identity given in \cite[Theorem 1.1]{BJ}
\begin{align*}
	F_\ou^\ed(-q) = \frac{\left(-q^2;q^2\right)_\infty}{2(-q;q)_\infty} + \frac{\left(-q^2;q^2\right)_\infty}{(q;q)_\infty}\sum_{n\in\Z} \frac{(-1)^nq^\frac{3n(n+1)}{2}}{1+q^n}.
\end{align*}
Using \eqref{f-definition}, we see after a short calculation that
\begin{align*}
	F_\ou^\ed(q) = \frac{\left(-q^2;q^2\right)_\infty}{2}\left(2 - f(-q)+\frac{1}{(q;-q)_\infty}\right).
\end{align*}
Now we combine \eqref{E:(-q^2;q^2)} with \eqref{E:(q;-q)} and \Cref{L:f lemma} to get
\begin{align*}
	F_\ou^\ed(q) \sim \frac12\sqrt\frac{\pi}{2z}e^\frac{\pi^2}{12z}
\end{align*}
as $z \to 0$ in any region $R_\Delta$. Applying \Cref{P:CorToIngham} with $\l=\frac12\sqrt{\frac\pi2}$, $\beta = - \frac 12$, and $\gamma = \frac{\pi^2}{12}$ gives the result.

\subsection{Proof of equation (\ref{T:OD-ED main term})}

By \Cref{Injection} we know that $p^\ed_\od(n)$ is weakly increasing, therefore we may apply \Cref{P:CorToIngham} to compute the asymptotics for $p^\ed_\od(n)$. 
By \cite[Theorem 1.1]{BJ}, the generating function
\begin{multline*}
F_\od^\ed(q) = 
1 + q + q^2 + 2 q^3 + 2 q^4 + 2 q^5 + 3 q^6 + 4 q^7 + 5 q^8 + 5 q^9 + 
 6 q^{10} \\ + 7 q^{11} + 9 q^{12} + 10 q^{13} + 12 q^{14} + 14 q^{15} + 16 q^{16} + 
 17 q^{17} + 20 q^{18} + \ldots
\end{multline*}
satisfies the identity
\begin{equation*}
	F_\od^\ed(q) = \dfrac{(1+q)\left( -q^2; q^2 \right)_\infty}{1-q} - \dfrac{q \left( -q; q^2 \right)_\infty}{1-q}.
\end{equation*}
Using \eqref{E:(-q^2;q^2)} we have that as $z \to 0$ in regions $R_\Delta$, 
\[
	\frac{(1+q)\left(-q^2;q^2\right)_\infty}{1-q} \sim \frac{\sqrt2e^\frac{\pi^2}{24z}}{z}.
\]
Likewise, using equation \eqref{E:(-q;q^2)} we obtain the asymptotic
\[
	-\frac{q\left(-q;q^2\right)_\infty}{1-q} \sim -\frac{e^\frac{\pi^2}{24z}}{z}
\]
as $z \to 0$ in $R_\Delta$. Combining these formulas, we have
\begin{align*}
	F^\ed_\od(q) \sim \dfrac{\sqrt{2}-1}{z} e^{\frac{\pi^2}{24z}}.
\end{align*}
Thus, by \Cref{P:CorToIngham} with $\l=\sqrt2-1$, $\b=-1$, and $\g=\frac{\pi^2}{24}$ equation (\ref{T:OD-ED main term}) follows.

\section{Open questions}\label{openquestion}

The wide variety of different functions and different properties exhibited among partitions with parts separated by parity suggests many different directions to pursue open questions. The generating functions which appear in the proof of Theorem \ref{T:Main Theorem} include modifications and combinations of modular forms, mock modular forms, false modular forms, and false-indefinite theta functions related to Maass forms. For these modular objects, there are natural inclusions within vector-valued versions. It would be natural to ask whether such vector-valuedness contains combinatorial information related to partitions with parts separated by parity or other connected partition-theoretic phenomena. It would also be natural to ask about two-variable generalizations of these generating functions, and whether such generalizations exhibit nice structure from the perspective of Jacobi forms or $q$-series.

There is a wide variety of applications of partition with parts separated by parity and connections with other topics in partition theory to be found in the literature. For instance, there are connections with mock theta functions \cite{FT} and separable integer partition classes \cite{CHTW}. There are also modifications of these functions which exhibit congruences \cite{A2, FT}, and it would be natural to ask whether suitable modifications of all of these functions exhibit congruences.

We also note that Theorem \ref{T:Main Theorem} implies for $n \gg 1$ the inequalities
\begin{align*}
	p_\ed^\od(n) < p_\od^\ed(n) < p_\od^\eu(n) < p_\ed^\ou(n) < p_\eu^\od(n) < p_\eu^\ou(n) < p_\ou^\ed(n) < p_\ou^\eu(n).
\end{align*}
Several of these inequalities have straightforward combinatorial explanations, but certain inequalities seem to lack obvious explanations. It would be interesting to obtain a more complete combinatorial explanation for why this string of inequalities holds.

\end{document}